\documentclass[11pt]{amsart}
\usepackage{amsthm,amsmath,amsxtra,amscd,amssymb,color,cleveref}

\newcommand{\al}{\alpha}
\newcommand{\be}{\beta}
\newcommand{\de}{\delta}
\newcommand{\ep}{\varepsilon}

\newcommand\chars[2]{\left[\begin{smallmatrix}#1\\ #2\end{smallmatrix}\right]}
\newcommand\tc[2]{\theta\chars{#1}{#2}}

\newcommand{\CC}{{\mathbb{C}}}

\newcommand{\PP}{{\mathbb{P}}}

\newcommand{\ZZ}{{\mathbb{Z}}}

\newcommand{\calH}{{\mathcal H}}

\newcommand{\calI}{{\mathcal I}}
\newcommand{\calA}{{\mathcal A}}

\newcommand{\calM}{{\mathcal M}}
\newcommand{\calJ}{{\mathcal J}}

\newcommand{\calD}{{\mathcal D}}
\newcommand{\calR}{{\mathcal R}}

\newcommand{\op}{\operatorname}

\newcommand{\Sp}{\op{Sp}}

\newcommand\diag{\operatorname{diag}}

\theoremstyle{plain}
\newtheorem{thm}{Theorem}
\newtheorem*{mthm*}{Main Theorem}
\newtheorem{lm}[thm]{Lemma}
\newtheorem{prop}[thm]{Proposition}
\newtheorem{cor}[thm]{Corollary}

\theoremstyle{definition}

\newtheorem{rem}[thm]{Remark}
\newtheorem{ntn}[thm]{Notation}

\usepackage{color}
\title{An explicit solution to the weak Schottky problem}
\author[H.M. Farkas]{Hershel M. Farkas}
\address{Department of Mathematics, Hebrew University, Jerusalem 91904, Israel}
\email{farkas@ma.huji.ac.il}
\author[S. Grushevsky]{Samuel Grushevsky}
\address{Mathematics Department, Stony Brook University, Stony Brook, NY 11794-3651, USA}
\email{sam@math.stonybrook.edu}
\thanks{Research of the second author is supported in part by National Science Foundation under the grant DMS-18-02116.}
\thanks{Research of the third  author  is  very partially supported by PRIN 2015 Moduli spaces and Lie Theory and progetto di ateneo 2015 Moduli, deformazioni e superfici K3}
\author[R. Salvati Manni]{Riccardo Salvati Manni}
\address{Dipartimento di Matematica, Piazzale Aldo Moro, 2, I-00185 Roma, Italy}
\email{salvati@mat.uniroma1.it}

\begin{document}
\begin{abstract}
We give an explicit weak solution to the Schottky problem, in the spirit of Riemann and Schottky. For any genus $g$, we write down a collection of polynomials in genus $g$ theta constants, such that their common zero locus contains the locus of Jacobians of genus $g$ curves as an irreducible component. These polynomials arise by applying a specific Schottky-Jung proportionality to an explicit collection of quartic identities for genus $(g-1)$ theta constants.
\end{abstract}
\maketitle
\section{Introduction}
We work throughout over the field of complex numbers. Our main result is the following explicit weak solution to the classical Schottky problem.
\begin{mthm*} \label{main}
For any $g\ge 4$, let $S_{34}$ denote the following degree $2^{3\cdot 2^{g-4}+1}$ polynomial in genus $g$ theta constants, evaluated at some period matrix $\tau$:
$$
\prod_{\substack{a_\ep,b_\ep,c_\ep=\pm 1\\ a_{0,\ldots,0}=1}}\ \sum_{\ep \in(\ZZ/2\ZZ)^{g-4}}
a_\ep\left(\tc{E & 0 & 0 & 0&\ep }{0 & 0 & 0 & 0&{\bf 0} }\tc{E & 0 & 0 & 0&\ep }{ 1 & 1 & 1 & 1&{\bf 1} }
\tc{E & 0 & 1 & 1&\ep }{0 & 1 & 0 & 0&{\bf 0} }\tc{E & 0 & 1 & 1&\ep }{1 & 0 & 1 & 1&{\bf 1} }\right.\cdot
$$
\vskip-4mm
$$
\begin{aligned}
&\left.\cdot\,\tc{1+E & 1 & 0 & 0&\ep }{0 & 0 & 0 & 1&{\bf 0} }\tc{1+E & 1 & 0 & 0&\ep }{ 1 & 1 & 1 & 0&{\bf 1} }
\tc{1+E & 1 & 1 & 1&\ep }{0 & 1 & 0 & 1&{\bf 0} }\tc{1+E & 1 & 1 & 1&\ep }{1 & 0 &  1 &  0&{\bf 1} }\right)^{1/2}\\
+&b_\ep\left(\tc{1+E & 0 & 1 & 0&\ep }{0 & 0 & 0 & 0&{\bf 0} }\tc{1+E & 0 & 1 & 0&\ep }{1 & 1 & 1 & 1&{\bf 1} }
\tc{1+E &  0 & 0 & 1&\ep }{0 & 1 & 0 & 0&{\bf 0} }\tc{1+E & 0 & 0 & 1&\ep }{1 & 0 & 1 & 1&{\bf 1} }\right.\cdot\\
&\ \ \left.\cdot\,\tc{E & 1 & 1 & 0&\ep }{0 & 0 & 0 & 1&{\bf 0} }\tc{E & 1 & 1 & 0&\ep }{1  & 1 & 1 & 0&{\bf 1} }
\tc{E & 1 & 0 & 1&\ep }{0 & 1 & 0 & 1&{\bf 0} }\tc{E & 1 & 0 &1&\ep }{1 & 0 & 1 & 0&{\bf 1} } \right)^{1/2}\\
+&c_\ep \left(\tc{E & 0 & 0 & 0&\ep }{0 & 0 & 1 & 1&{\bf 0} }\tc{E & 0 & 0 & 0&\ep }{1 & 1 & 0 & 0&{\bf 1} }
\tc{E & 0 & 1 & 1&\ep }{0 & 1 & 1 & 1&{\bf 0} }\tc{E & 0 & 1 & 1&\ep }{1 & 0 & 0 & 0&{\bf 1} }\right.\cdot\\
&\ \ \left.\cdot\,\tc{1+E & 1 & 0 & 0&\ep }{0 & 0 & 1 & 0&{\bf 0} }\tc{1+E & 1 & 0 & 0&\ep }{1 & 1 & 0 & 1&{\bf 1} }
\tc{1+E & 1  & 1 & 1&\ep }{0 & 1  & 1 & 0&{\bf 0} }\tc{1+E & 1 & 1 & 1&\ep }{1 & 0 & 0 & 1&{\bf 1} }\right)^{1/2},
\end{aligned}
$$
where $E:=\ep_1+\ldots+\ep_{g-4}\in\ZZ/2\ZZ$.
For any $3\le j<k\le g$ let $S_{jk}$ be obtained from $S_{34}$ by swapping columns $3$ and $j$, and columns $4$ and $k$ of the characteristics of all theta constants appearing in the expression.

Then the collection of equations $\lbrace S_{jk}=0\rbrace_{3\le j<k\le g}$ gives a weak solution to the Schottky problem, i.e.~the common zero locus of the modular forms $\lbrace S_{jk}\rbrace_{3\le j<k\le g}$ contains the Jacobian locus as an irreducible component.
\end{mthm*}
Here, and throughout the paper, we write $\bf 0$ and $\bf 1$ for strings of zeroes or ones of appropriate length.

\medskip
One can say that the theory we deal with here began with Riemann's papers~\cite{riemann1} and~\cite{riemann2}. In~\cite{riemann2} in particular it seems clear that Riemann was working towards understanding what we now think of as the Schottky problem, even though the main immediate application was a proof of the Jacobi inversion theorem. The field then blossomed, with a flurry of activity by A.~Krazer, W.~Wirtinger, M.~Noether, F.~Schottky, G.~Frobenius, H.~Baker, and many others --- see the many references in~\cite{rafabook}. In the middle of the 20th century the interest in the Schottky problem seems to have waned, probably due to the fact that
the length of the identities increased exponentially, and not much new was discovered on the classical Schottky problem. The interest in the  subject was then rekindled in the 1970s in particular by Rauch's rediscovery  of~\cite{scju}, where what are now called the Schottky-Jung proportionalities were stated without proof. These were proven rigorously by the first author~\cite{fasc}, and their connection with the Schottky problem was discussed in~\cite{farascju}. A period of intense activity followed, including Mumford's development of the algebraic theory of the theta function, and integrable systems entering the picture.

Various approaches to the Schottky problem were developed by the 1980s, and various geometric solutions to the problem were then obtained (surveyed eg.~in~\cite{vgsurvey,grschottky}). Igusa~\cite{igusagen4} and Freitag~\cite{freitaggen4} showed that Schottky's original equation~\cite{schottky} gives indeed the solution to the Schottky problem in genus 4. Accola~\cite{acc} gave a collection of equations in theta constants that characterize the Jacobian locus in genus 5, up to additional irreducible components.
Arbarello and De Concini~\cite{adc} showed that there exists a finite set of equations in theta constants and their derivatives that characterize Jacobians (however, making them explicit requires elimination of $3g$ complex numbers from a system of equations). Shiota~\cite{shiota} proved Novikov's conjecture characterizing Jacobians by their theta function satisfying the KP equation, and various other approaches were developed.

In a spirit closest to the current paper, van Geemen~\cite{vgeemenscju} and Donagi~\cite{donagiscju} showed that the classical Schottky-Jung approach gave a weak solution to the Schottky problem --- however, their results do not lead to explicit equations, as we will explain shortly.

More recently, Krichever~\cite{krichevertrisecant} proved the celebrated Welters'~\cite{welters} trisecant conjecture, characterizing Jacobians by their Kummer varieties having trisecant lines. We hope that our explicit solution of the weak Schottky problem, motivated by the viewpoint of~\cite{rafabook}, may lead to a further rejuvenation of interest in this classical subject.

\medskip
We now state the Schottky problem more precisely, and motivate our main theorem. Denote by $\calM_g$ the moduli space of curves of genus $g$, denote by $\calA_g$ the moduli space of complex principally polarized abelian varieties (ppav) of dimension $g$,
and denote $J:\calM_g\to\calA_g$ the Torelli morphism. The {\em Schottky problem} is to  characterize {\em the locus of Jacobians} $\calJ_g$, which is defined to be the closure of $J(\calM_g)$ in $\calA_g$.

The {\em classical Schottky problem}, studied by Riemann and Schottky, is to write down the defining equations for$\calJ_g$. More precisely, recall that theta constants with characteristics define an embedding $Th:\calA_g(4,8)\hookrightarrow \PP^{2^{g-1}(2^g+1)-1}$ of the level cover of $\calA_g$ (see the next section for details), and the classical Schottky problem is to determine the defining ideal $\calI_g^J$ of $Th(\calJ_g(4,8))\subset Th(\calA_g(4,8))$. The {\em weak Schottky problem} is the problem of characterizing the locus of Jacobians up to extra irreducible components. Classically, this means finding an ideal $I_g$  of polynomials in theta constants, such that the zero locus of $I_g$ within $Th(\calA_g(4,8))$ contains $Th(\calJ_g(4,8))$ as an irreducible component.

The Schottky problem is non-trivial for $g\ge 4$, and Schottky's original equation solves the classical Schottky problem in genus 4, as discussed above. Despite many approaches to the Schottky problem having been developed, a solution to the classical Schottky problem and its weak version have remained elusive for any genus $g \ge 5$.

\smallskip
The Schottky-Jung proportionalities (reviewed in section 3) relate theta constants of a genus $g$ Jacobian to the theta constants of its Prym, which is a $(g-1)$-dimensional ppav. More precisely, denote by $\calR_g$ the moduli space of connected unramified double covers of smooth genus $g$ curves, thought of as pairs consisting of a curve $C\in\calM_g$ and a non-zero two-torsion point $\eta$ on the Jacobian of $C$. The Prym construction is the map $Pr:\calR_g\to\calA_{g-1}$, and  the Schottky-Jung proportionalities relate the theta constants of $Pr(C,\eta)$ and of $C$.

Let $\calI_g^A$ denote the defining ideal of $Th(\calA_g(4,8))\subset \PP^{2^{g-1}(2^g+1)-1}$. Applying the Schottky-Jung proportionalities to each theta constant appearing in a polynomial $P\in\calI_{g-1}^A$ yields an element $SJ^\eta(P)\in\calI_g^J$, which we will call the corresponding Schottky-Jung identity. The {\em big Schottky} locus is the locus within $Th(\calA_g(4,8))$ defined by the equations $SJ^\eta(P)$ for all $P\in\calI_{g-1}^A$, for {\em one fixed} $\eta$, while the {\em small Schottky} locus is the locus defined by such equations for {\em all possible} $\eta$. Van Geemen~\cite{vgeemenscju} and Donagi~\cite{donagiscju} showed that the small and the big Schottky loci, respectively, give weak solutions to the Schottky problem, while in~\cite{donagiintermjac} Donagi showed that already in genus 5 the big Schottky locus contains an extra irreducible component, containing the locus of intermediate Jacobians of cubic threefolds. In the unpublished preprint~\cite{siegel} it is shown that in genus 5 the small Schottky locus is in fact equal to the Jacobian locus.

Note, however, that for $g\ge 3$ the ideal $\calI_g^A$ of relations among the theta constants of a general ppav is not known --- it is conjectured that it is generated by Riemann's quartic relations (see~\cite{fsm} for details), but no approaches to proving this are available. Thus for any $g\ge 5$ the results of van Geemen and Donagi cannot be made explicit, as one cannot write down a set of generators of $\calI_{g-1}^A$.
Our main theorem is thus the first known explicit weak solution to the classical Schottky problem. After our paper appeared on the arXiv, our equations were investigated numerically in genus 5 by Agostini and Chua~\cite{agch}.

\smallskip
Our equations $S_{jk}$ arise by applying the Schottky-Jung proportionalities to certain quartic identities in theta constants. We note, however, that while the usually applied case of the Schottky-Jung proportionalities is for the two-torsion point $\eta_0:=\chars{0&0&\ldots&0}{1&0&\ldots&0}$, it turns out (see~\Cref{rem:etag}) that for our methods we need to use the two-torsion point $\eta_g:=\chars{0&\ldots&0}{1&\ldots&1}$. The polynomials $S_{jk}$ arise by applying the Schottky-Jung proportionalities for $\eta_g$ to the following quartic identity in theta constants.
\begin{prop}\label{prop:Rsum}
For any $g\ge 3$, let $R_{34}$ be the following quartic polynomial in theta constants, all evaluated at some period matrix $\tau$:
\begin{equation}\label{eq:Rgeng}
\begin{aligned}
R_{34}:=\sum_{\ep\in(\ZZ/2\ZZ)^{g-3}}
&\tc{0 & 0 & 0&\ep}{0 & 0 & 0&{\bf 0}}\tc{0 & 1 & 1&\ep}{1 & 0 & 0&{\bf 0}}
\tc{1 & 0 & 0&\ep}{0 & 0 & 1&{\bf 0}}\tc{1 & 1 & 1&\ep}{1 & 0 & 1&{\bf 0}}\\
-&\tc{0 & 1 & 0&\ep}{0 & 0 & 0&{\bf 0}}\tc{0 & 0 & 1&\ep}{1 & 0 & 0&{\bf 0}}
\tc{1 & 1 & 0&\ep}{0 & 0 & 1&{\bf 0}}\tc{1 & 0 & 1&\ep}{1 & 0 & 1&{\bf 0}} \\
+&\tc{0 & 0 & 0&\ep}{0 & 1 & 1&{\bf 0}}\tc{0 & 1 & 1&\ep}{1 & 1 & 1&{\bf {\bf 0}}}
\tc{1 & 0 & 0&\ep}{0 & 1 & 0&{\bf 0}}\tc{1  & 1 & 1&\ep}{1  & 1 & 0&{\bf 0}}.
\end{aligned}
\end{equation}
Let $R_{jk}$ be obtained by permuting columns $2$ and $j-1$, and $3$ and $k-1$ in the expression of $R_{34}$. Then each $R_{jk}$ vanishes identically in $\tau$, i.e.~$R_{jk}\in\calI_g^A$.
\end{prop}
This Proposition is an immediate corollary of the generalized Riemann relations stated in ~\cite{fayriemann}. See~\Cref{rem:noRiemann} for an explanation why the classical Riemann quartic relation cannot be used instead of $R_{34}$ for our argument.

The Schottky-Jung proportionalities for~$\eta_g$ are given explicitly by~\eqref{eq:SJetag}, and since $\eta_g$ is invariant under any permutation of columns (and this is one reason for our choice of $\eta_g$ rather than $\eta_0$), we obtain~\Cref{lm:SR}: the statement that $S_{jk}=SJ^{\eta_g}(R_{jk})$ for any $3\le j<k\le g$. As a corollary, we thus obtain an explicit proof of the main result of Donagi's paper~\cite{donagiscju}, which implies the main result of van Geemen's paper~\cite{vgeemenscju}:
\begin{cor}
The big Schottky locus gives a weak solution to the Schottky problem, i.e.~the common zero locus of $SJ^{\eta_g}(R)$, for all $R\in\calI_{g-1}^A$, contains $\calJ_g$ as an irreducible component.
\end{cor}
\noindent (Of course we have in fact proven that it is enough to take Riemann's quartic relations, as a subset of $\calI_{g-1}^A$, and among those, to only take those that imply all $R_{jk}$).

We prove the main theorem by expanding $S_{jk}$ near the locus of diagonal period matrices, and showing that the lowest degree terms of the expansions give a collection of what are called Poincar\'e relations. These are infinitesimal (i.e.~up to terms of higher order) relations satisfied by period matrices of Jacobians that are close to being diagonal. The only such relation in genus 4 was proven by Poincar\'e~\cite{poincare}. Poincar\'e states in~\cite[pp.~298--299]{poincare} that his relation generalizes to arbitrary genus. Our work gives  a complete proof of Poincar\'e relations for arbitrary Riemann surfaces of arbitrary genus. A different proof is given by Fay~\cite[Ex.~3.4, p.~45]{faytheta}. Rauch in~\cite[p.~228]{rafabook} asks whether in general one can find Schottky(-Jung) identities that imply Poincar\'e relations, and in particular our work answers this question in the affirmative. By analyzing these lowest order terms of $S_{jk}$, we then show that in a neighborhood of a generic diagonal period matrix these equations are functionally independent, and thus that the common zero locus of all $S_{jk}$ is $(3g-3)$-dimensional, which implies that it contains the Jacobian locus as an irreducible component.

\smallskip
The structure of the text is as follows. In Section 2 we fix the notation and review and extend the results of Fay~\cite{fayriemann} and the third author~\cite{smrelations} on the linear span of Riemann's quartic relations, and show that the identity $R_{34}$ among theta constants holds.
In Section 3 we recall the Schottky-Jung proportionalities, and give the explicit formula~\eqref{eq:SJetag} for them for the two-torsion point $\eta_g$. In Section 4 we briefly recall the well-known expansion of theta constants near the locus of diagonal period matrices.
In Section 5 we recall Poincar\'e's notion of ``infinitesimal" relations for periods of Riemann surfaces near diagonal matrices, and prove~\Cref{thm:Poincareholds}, proving the generalization of his relations to arbitrary genus. The proof is by looking at the lowest order terms of the expansions of our $S_{jk}$. Finally, in Section 6 we combine all of these ingredients to prove that $S_{jk}$ are locally functionally independent near $\calD_g$, which implies our Main Theorem.

\subsection*{Acknowledgements}
We are grateful to G.~Codogni and B.~van Geemen for interesting discussions and useful comments on the manuscript.

\section{Riemann's quartic relations and their linear combinations}
In this section we fix the notation for moduli spaces of curves, abelian varieties, and their covers. We then recall the theta constants, their properties, and relations among them, for the details on all of this we refer to~\cite{igusabook} and \cite{rafabook}.

We denote by $\calH_g:=\lbrace\tau\in\operatorname{Mat}_{g\times g}(\CC):\tau^t=\tau; \operatorname{Im} \tau>0\rbrace$ the Siegel upper half-space. The quotient $\calA_g:=\calH_g/\Sp(2g,\ZZ)$ is the moduli space of complex principally polarized abelian varieties (ppav).

For any even $\ell$, let $\Gamma_g(\ell)\subset\Sp(2g,\ZZ)$ be the normal subgroup that is the kernel of the map to $\Sp(2g,\ZZ/\ell\ZZ)$, and let $\Gamma_g(\ell,2\ell)$ be the subgroup of $\Gamma_g(\ell)$ consisting of matrices such that the diagonals of $A^t B$ and $C^t D$ (where $\gamma\in\Sp(2g,\ZZ)$ is written in block form $\gamma=\left(\begin{smallmatrix}A&B\\ C&D\end{smallmatrix}\right)$) are congruent to zero modulo $2\ell$. The level covers of moduli of ppav are then the quotients $\calA_g(\ell):=\calH_g/\Gamma_g(\ell)$ and $\calA_g(\ell,2\ell):=\calH_g/\Gamma_g(\ell,2\ell)$ by these level subgroups.

Given $\ep,\de\in\ZZ^g$, the theta function with characteristics $ \chars\ep\de$  is defined as
$$
 \tc\ep\de(\tau,z):=\sum_{n\in\ZZ^g}\exp\left(\pi i (n+\ep/2)^t\left(\tau (n+\ep/2)+z+\de\right)\right).
$$
The theta constant is the evaluation of theta function at $z=0$. We will be mostly concerned with theta constants, and will write simply $\tc\ep\de$ for such, if $\tau$ and $z=0$ is understood. We will always write the~$z$ variable when dealing with theta functions. The theta constants satisfy the identity $ \tc{\ep+2a}{\de +2b}(\tau,0)= (-1)^{\ep^t b} \tc\ep\de(\tau,0)$ for any $a,b\in\ZZ^g$, and thus (up to sign) it is often convenient to work with characteristics lying in $(\ZZ/2\ZZ)^{2g}$, but we will follow~\cite{faytheta} in working with characteristics in $\ZZ^{2g}$ which accounts for some sign differences between our formulas and those in the literature.
By  an  abuse of notation, we will write $\ep,\de$ as row vectors, but they will be treated as column vectors, as will be all vectors used in calculations. We call $\ep$ the top, and $\de$ the bottom characteristic, and say $(g)$-characteristic when we want to emphasize the dimension in which we are working. A characteristic is called even or odd depending on whether $e(\chars\ep\de):=\ep^t\cdot\de$ is even or odd, correspondingly. All theta constants with  odd characteristics vanish identically in $\tau$. \begin{ntn}
For convenience, we denote by $K_g=(\ZZ/2\ZZ)^{2g}$ the set of characteristics, denote by $K_g^+,K_g^-\subset K_g$ the sets of even and odd characteristics, respectively, and let $k_g^\pm:= 2^{g-1}(2^g\pm 1)$ be the cardinalities of the sets $K_g^\pm$.
\end{ntn}

Defining the action of $\Sp(2g,\ZZ)$ on characteristics via
$$
 \gamma\circ\chars\ep\de:=\left(\begin{smallmatrix} D& -C\\ -B& A\end{smallmatrix}\right)\chars\ep\de+\left[\begin{smallmatrix}\diag(CD^t)\\ \diag(AB^t)\end{smallmatrix}\right],
$$
theta constants satisfy the following transformation formula (see~\cite{igusabook}):
$$
\begin{aligned}
 \theta&\left[\gamma\circ\chars{\ep}{\de}\right](\gamma\circ\tau)=
 \kappa(\gamma)\sqrt{\det(C\tau+D)}\tc\ep\de(\tau)\\
  &\cdot \exp\left((-\pi i/4) \ep^t D^tB\ep - 2 \de C^t B\ep +\de^t C^t A\de -2(D\ep^t-C\de^t)\diag(AB^t)\right),
\end{aligned}
$$
where $\kappa$ is some eighth root of unity independent of the characteristic $\chars\ep\de$. It moreover turns out that $\kappa(\gamma)=1$ for any $\gamma\in\Gamma_g(4,8)$, and thus each theta constant is a modular form with respect to $\Gamma_g(4,8)$, which is to say that
$$
 \tc\ep\de(\gamma\circ\tau)=\sqrt{\det(C\tau+D)}\tc\ep\de(\tau)
$$
for any $\gamma\in\Gamma_g(4,8)$ (where the square root can in fact be chosen globally). The map sending a ppav to the set of all even theta constants then defines an embedding
$$
 Th:\calA_g(4,8)\hookrightarrow\PP^{2^{g-1}(2^g+1)-1}.
$$
The classical, and still unsolved, question of determining all relations among theta constants, is the question of determining the defining ideal $\calI_g^A$ of $Th(\calA_g(4,8))\subset \PP^{2^{g-1}(2^g+1)-1}$. The only known relations  among theta constants  are Riemann's quartic relations.  To write them, recall the Weil pairing of two characteristics $m=\chars{\ep_1}{\de_1}$ and $n=\chars{\ep_2}{\de_2}$ defined by
$$e(m,n):= (-1)^{ \ep_1^t\cdot\de_2 -  \ep_2^t\cdot\de_1}.$$

We recall from~\cite{fayriemann} the following special case of (generalized) Riemann's theta formula. Let $A\subset K_g$ be a  subgroup of order $2^{g-m}$  and let $B$ be the orthogonal group with respect to the symplectic form, i.e. $B:= \{m\in  K_g : e(m,n)=1 \,{\rm for\, all}\, n\in A\}$. Then the following identity holds:
\begin{equation}\label{grt}
\begin{aligned}
2^m&\sum_{\chars{\ep }{\de }\in A} e(\chars{\ep }{\de }) (-1)^{\de^t (\al+\sigma)} \tc\ep\de\tc{\ep+\al}{\de+\be}\tc{\ep+\sigma}{\de+\mu}\tc{\ep+\al+\sigma}{\de+\be+\mu}\\
&=\sum_{\chars{\ep }{\de }\in B} e(\chars{\ep }{\de }) (-1)^{\de^t (\al+\sigma)} \tc\ep\de\tc{\ep+\al}{\de+\be}\tc{\ep+\sigma}{\de+\mu}\tc{\ep+\al+\sigma}{\de+\be+\mu}.
\end{aligned}
\end{equation}

\begin{proof}[Proof of~\Cref{prop:Rsum}]
To prove that $R_{34}$ lies in $\calI_g^A$, we simply apply~\eqref{grt} for
$$A:=\{ \chars{0&0&0&\ep}{0&0&0&\bf{0} }\,  {\rm for\ all }\, \ep \in (\ZZ/2\ZZ)^{g-3}\},$$
$$\chars{\al}{\be}=\chars{0&1&1&0\ldots&0}{1&0&0&0\ldots&0}, \quad
\chars{\sigma}{\mu}=\chars{1&0&0&0\ldots&0}{0&0&1&0\ldots&0}.$$

Let $V_0$  be the  two-dimensional subspace  of $K_g$ spanned by $\chars{\al}{\be}, \chars{\sigma}{\mu}$, and set $V_1:=V_0+\chars{0&1&0&0\ldots&0}{0&0&0&0\ldots&0}$ , $V_2:=V_0+\chars{0&0&0&0\ldots&0}{0&1&1&0\ldots&0}$  and $V=V_0\cup V_1\cup V_2$.  Unless all theta constants appearing in the monomial in the sum in the right-hand-side of~\eqref{grt} are even, such a summand vanishes. Thus the only non-vanishing summands are those when $\chars{\ep}{\de}\in A+V$, i.e.~can be written as the sum of a characteristic lying in~$A$ and a characteristic lying in~$V$. 
Moreover, the summands in the right-hand-side are the same for all $\chars{\ep}{\de}\in  V_i +m$, for $m\in A$ fixed.

This gives, up to a factor $4$, precisely the expression for $R_{34}$ given by formula~\eqref{eq:Rgeng}, and thus $R_{34}$ lies in $\calI_g^A$. 
Furthermore, each $R_{jk}$ is obtained from $R_{34}$ by permuting some columns of the characteristics. This permutation can of course be obtained by the action of some element $\gamma$ of the  symplectic group,
 and since $\calI_g^A$ is invariant under $\Sp(2g,\ZZ)$, it follows that also $R_{jk}=\gamma\circ R_{34}\in\calI_g^A$.
\end{proof}
As an example (and to highlight how sign conventions differ depending on whether one writes characteristics as elements of $\ZZ$ or of $\ZZ/2\ZZ$), consider the case of genus $g=3$. The equation above then reads
 $$
 \begin{aligned}
  8\tc{0&0&0}{0&0&0}\tc{0&1&1}{1&0&0}\tc{1&0&0}{0&0&1}\tc{1&1&1}{1&0&1}=& 4\tc{0&0&0}{0&0&0}\tc{0&1&1}{1&0&0}\tc{1&0&0}{0&0&1}\tc{1&1&1}{1&0&1}\\
 4\tc{0&1&0}{0&0&0}\tc{0&2&1}{1&0&0}\tc{1&1&0}{0&0&1}\tc{1&2&1}{1&0&1}
 &+4\tc{0&0&0}{0&1&1}\tc{0&1&1}{1&1&1}\tc{1&0&0}{0&1&2}\tc{1&1&1}{1&1&2}\end{aligned}
 $$
 Converting to characteristics in $\ZZ/2\ZZ$, this gives the classical form of Riemann's quartic addition theorem, see~\cite{vgsurvey}:
 \begin{equation}\label{eq:Rgen3}
  \begin{aligned}
 \tc{0&0&0}{0&0&0}\tc{0&1&1}{1&0&0}\tc{1&0&0}{0&0&1}\tc{1&1&1}{1&0&1}&=  \\
  \tc{0&1&0}{0&0&0}\tc{0&0&1}{1&0&0}\tc{1&1&0}{0&0&1}\tc{1&0&1}{1&0&1}
 &-\tc{0&0&0}{0&1&1}\tc{0&1&1}{1&1&1}\tc{1&0&0}{0&1&0}\tc{1&1&1}{1&1&0}
 \end{aligned}
 \end{equation}
\begin{rem} \label{rem:doubling}
This genus 3 case of Riemann's quartic addition theorem can be considered as an alternative starting point for our constructions, using only Riemann's quartic addition theorem.

Consider the square matrix  $M(g)$ of  size $2^{2g}$  whose entries are the Weil pairings of all pairs of characteristics $m, n \in K_g$, which we write as
$$M(g):=\left(\begin{smallmatrix} M^+(g)&N(g)\\ N(g)^t &M^-(g)\end{smallmatrix}\right), $$
where the set of characteristics is ordered in such a way that $k_g^+$ even characteristics in the set $K_g^+$ appear before the  $k_g^-$ odd ones.  Recall that by~\cite{smrelations} all quartic identities in theta constants are  related to  the eigenvectors of eigenvalue $-2^{g-1}$ of the  matrix $M^+(g)$.

It turns out that there is a simple procedure, which we call the {\em doubling principle} that constructs these eigenvectors recursively from the eigenvectors of eigenvalue $-2^{g-2}$ of the  matrix $M^+(g-1)$. In fact, this $-2^{g-1}$-eigenspace of  $M^+(g)$ is equal to the direct sum of vector spaces $U_1\oplus U_2\oplus U_3\oplus U_4$, where
\begin{itemize}
\item $U_1$ and $U_2$ are spanned by the vectors of the form
$$
  u_1:=(X,X,0,0)^t ,\quad \mbox {and}\quad u_2:=(X,0,X,0)^t,
$$
respectively, for $X$ being any $-2^{g-2}$-eigenvector of $M^+(g-1)$;
\item $U_3$ is   spanned by the vectors of the form $u_3:=(X,0,0,Y)^t$,
for  $\left(\begin{smallmatrix}X\\ Y\end{smallmatrix}\right)$ being any $-2^{g-1}$-eigenvector of $M(g-1)$; and
\item $U_4$ is   spanned by the vectors of the form $u_4:=(X,-X,-X,0)^t$, for $X$ being any $2^{g-1}$-eigenvector of $M^+ (g-1)$.
\end{itemize}
It turns out then that the quartic relation $R_{34}$ in any genus can be obtained by repeatedly applying this doubling principle starting from to \eqref{Rgen3} in genus 3, which gives an alternative proof of~\Cref{prop:Rsum}.
\end{rem}

\section{The Schottky-Jung proportionalities}
We denote by $\calM_g(4,8)$ the fiber product of $\calM_g$ and $\calA_g(4,8)$ over $\calA_g$ under the Torelli map and the forgetful morphism. By abuse of notation, we denote $J:\calM_g(4,8)\to\calA_g(4,8)$ also the lift of the Torelli map.

The Schottky-Jung proportionalities relate the theta constants of the Jacobian and of the Prym. They were discovered in~\cite{scju}, rigorously proven to hold in~\cite{fasc,farascju}, and recast algebraically by Mumford in~\cite{mumfordprym}. See the surveys ~\cite{dosurvey},\cite{beprymsurvey}, for the details of the Schottky-Jung approach, and~\cite{faprymsurvey} for a survey on Pryms. The Schottky-Jung proportionalities for arbitrary two-torsion point $\eta$, as a relation between modular forms, is described explicitly in

Denote by $\calR_g$ the moduli space of pairs $(C,\eta)$, where $C\in\calM_g$, and $\eta\in J(C)[2]\setminus\{0\}$ is a non-zero 2-torsion point on the Jacobian. Such a point $\eta$ defines an unramified connected double cover $\tilde C\to C$, and the Prym $Pr(C,\eta)$ is then defined to be the connected component of the kernel of the map $J(\tilde C)\to J(C)$. The Prym turns out to have a natural principal polarization, so that the construction defines a morphism $Pr:\calR_g\to\calA_{g-1}$. For the two-torsion point
\begin{equation}\label{eq:eta0}
\eta_0:=\chars{0&0&\ldots& 0}{1&0&\ldots&0},
\end{equation}
the {\em Schottky-Jung proportionalities} are the equalities
\begin{equation}\label{eq:SJ}
  \theta^2\chars\ep\de(Pr(C,\eta_0))=c\,\tc{0&\ep}{0&\de}(J(C))\tc{0&\ep}{1&\de}(J(C)),
\end{equation}
which hold for some non-zero constant $c$ independent of $\chars\ep\de$.

For our purposes, the combinatorics is such that we will need the explicit form of the Schottky-Jung proportionalities for the two-torsion point
\begin{equation}\label{eq:etag}
 \eta_g:=\chars{0&0&\ldots& 0}{1&1&\ldots&1}.
\end{equation}
In~\cite{farkashscju} the case of Schottky-Jung proportionalities for $\eta'=\chars{0&0&0&\ldots&0}{1&1&0&\ldots&0}$ was studied explicitly, and
and it turns out that additional signs depending on $\chars\ep\de$ appear. We will not be able to track the signs anyway, and only the characteristics appearing in the proportionality will play a role. The Schottky-Jung proportionalities for arbitrary $\eta$, as relations between modular forms, are described explicitly by van Geemen in~\cite{vgsurvey}, as we now quickly recall.

Since $\Sp(2g,\ZZ/2\ZZ)$ acts transitively on $J(C)[2]\setminus\lbrace 0\rbrace$, acting on the Schottky-Jung proportionalities~\eqref{eq:SJ} for $\eta_0$ by an element of the symplectic group that sends $\eta_0$ to $\eta$, one obtains the general case of the Schottky-Jung proportionalities. For a given $\eta$ one has an isomorphism $j:Pr(C,\eta)[2]\to V/\eta$, where $Pr(C,\eta)[2]$ is the set of two-torsion points of the Prym, and $V=\eta^\perp\subset K_g$ is the set of all $v$ such that the symplectic pairing $e(\eta,v)=0$. Then for any $\chars\ep\de\in Pr(C,\eta)[2]$, the characteristics appearing in the right-hand-side of Schottky-Jung proportionalities for~$\eta$ would be $j\left(\chars\ep\de\right)$ and $\eta+j\left(\chars\ep\de\right)$.

Note that the proportionalities depend on the choice of $j$; as this was never explained in the literature, we make this precise. A choice of a basis of $J(\tilde C)[2]$ gives the lifting (denoted $\pi^*$ by Mumford) $Pr(C,\eta)[2]\to J(\tilde C)[2]$, and $j$ is then obtained by composing this with projection to $J(C)[2]$. Thus  any choice of $j$ can arise; the standard choice for $\eta_0$ is to embed $(\ZZ/2\ZZ)^{2g-2}$ into $K_g$ as the characteristics with first column equal to $\chars00$.

For our choice of $\eta_g$, we will choose $j$ to be
$$
 j\left(\chars{\ep_1&\ldots&\ep_{g-1}}{\de_1&\ldots&\de_{g-1}}\right):=\chars{\ep_1+\ldots+\ep_{g-1}&\ep_1&\ldots&\ep_{g-1}}{0&\de_1&\ldots&\de_{g-1}}.
$$
Then the Schottky-Jung proportionalities take the explicit form
\begin{equation}\label{eq:SJetag}
 \!\! \theta^2\chars\ep\de(Pr(C,\eta_g))=\pm c\,
  \tc{\sum\ep_i&\ep_1&\ldots&\ep_{g-1}}{0&\de_1&\ldots&\de_{g-1}}(J(C))\,\cdot\,
  \tc{\sum\ep_i&\ep_1&\ldots&\ep_{g-1}}{1&1+\de_1&\ldots&1+\de_{g-1}}(J(C)),
\end{equation}
which we will use from now on --- where the sign depends on $\chars\ep\de$.

\medskip
The classical approach to the Schottky problem is as follows. Given any equation $P\in\calI_{g-1}^A$, note that $P$ is satisfied by the theta constants of $P(C,\eta)$ for any $(C,\eta)\in\calR_g$. Replacing in the polynomial $P$ each theta constant of $Pr(C,\eta)$ by the square root of the product of the two corresponding theta constants of $J(C)$ given by the Schottky-Jung proportionalities gives a polynomial in the {\em square roots} of genus $g$ theta constants of $J(C)$. Since the square roots cannot be chosen globally, to make this a well-defined relation we multiply the product of the results of such substitutions for {\em all possible choices} of the values of square roots of the monomials involved --- except choosing one square root of one monomial to be given (or otherwise one obtains each factor twice, once with plus and once with minus sign). We denote this product by $SJ^{\eta}(P)$, and call it the {\em Schottky-Jung identity} corresponding to $P$ and $\eta$. It is then a polynomial in genus $g$ theta constants, of degree equal to the degree of $P$, times $2$ raised to the power equal to the number of monomials in $P$ minus one,  and Schottky-Jung proportionalities imply that $SJ^{\eta}(P)\in\calI_g^J$.

The simplest non-trivial case of this is for $g=4$: the genus 3 Riemann's quartic relation involves three monomials and has the form $r_1-r_2+r_3=0$, where each $r_i$ is a product of four genus 3 theta constants. 
Applying the Schottky-Jung proportionalities
gives $\sqrt{R_1}-\sqrt{R_2}-\sqrt{R_3}=0$, where each $R_i$ is a product of eight theta constants evaluated at $ \tau \in J(C)$. To obtain a polynomial in theta constants, one needs to take the product
$(\sqrt{R_1}+\sqrt{R_2}+\sqrt{R_3})(\sqrt{R_1}-\sqrt{R_2}+\sqrt{R_3})
(\sqrt{R_1}+\sqrt{R_2}-\sqrt{R_3})(\sqrt{R_1}-\sqrt{R_2}-\sqrt{R_3})$, which is equal to
\begin{equation}\label{eq:Rconjugates}
SJ^\eta(r_1-r_2+r_3)=R_1^2+R_2^3+R_3^2-2R_1R_2-2R_1R_3-2R_2R_3=0,
\end{equation}
which is a degree 16 {\em polynomial} in theta constants of the genus 4 Jacobian $J(C)$. This equation was discovered by Schottky, and was rigorously proven by Igusa~\cite{igusagen4}  and Freitag~\cite{freitaggen4} to generate $\calI_4^J$, i.e.~to be the unique defining equation for $Th(\calJ_4)\subset Th(\calA_4)$.

In our case, applying Schottky-Jung proportionalities for $\eta_g$ to $R_{jk}$ gives
\begin{lm}\label{lm:SR}
For any $3\le j<k\le g$ the identity $S_{jk}=SJ^{\eta_g}(R_{jk})$ holds.
\end{lm}
\begin{proof}
Indeed, applying the Schottky-Jung proportionalities~\eqref{eq:SJetag} to $R_{34}$, one obtains an expression in square roots of theta constants that is one of the factors in $S_{34}$, and then taking the product over all possible choices of square roots gives exactly the product given by $S_{34}$, so that $S_{34}=SJ^{\eta_g}(R_{34})$. Since the columns number $3,\ldots,g$ of $\eta_g$ are all equal, permuting the columns $3$ and $j$, and columns $4$ and $k$ gives the desired equality for all $j,k$.
\end{proof}

\section{Expansion of theta constants near the diagonal}
The locus $\calJ_g$ contains the locus of ppav that are products of $g$ elliptic curves. Explicitly, we think of this locus as the image in $\calA_g$ of the locus $\calD_g\subset\calH_g$ consisting of diagonal period matrices $\calD_g:=\lbrace \tau=\diag(t_{1},t_{2},\ldots,t_{g})\rbrace$.
The proof of the main theorem will consist of showing that the $S_{jk}$ are locally functionally independent near $\calD_g$, by computing the lowest order term of their expansions.  We thus recall the expansion of theta constants near the diagonal, which was used  by Rauch~\cite{rauchschpo} in his genus 4 computations.

First recall that the theta constant of a diagonal period matrix decomposes as a product:
$$
 \tc\ep\de(\diag(t_{1},t_{2},\ldots,t_{g}))=\tc{\ep_1}{\de_1}(t_1)\cdot\ldots\cdot\tc{\ep_g}{\de_g}(t_g).
$$
Recall further for any $j<k$ the heat equation
$$
 \frac{\partial\tc\ep\de}{\partial\tau_{jk}}=\frac{1}{2\pi i}\frac{\partial^2\tc\ep\de}{\partial z_jz_k}.
$$
We then evaluate for any $j<k$ the partial derivative
$$\begin{aligned}
 \frac{\partial\tc\ep\de}{\partial\tau_{jk}}|_{\tau=\diag(t_{1},t_{2},\ldots,t_{g})}&=\frac{1}{2\pi i}
 \tc{\ep_1}{\de_1}(t_1)\cdot\ldots\cdot\tc{\ep_{j-1}}{\de_{j-1}}(t_{j-1})\\
 &\!\!\!\!\!\!\cdot\frac{\partial\tc{\ep_j}{\de_j}(t_j,z)}{\partial z}|_{z=0}\cdot\tc{\ep_{j+1}}{\de_{j+1}}(t_{j+1})\cdot\ldots\cdot\tc{\ep_{k-1}}{\de_{k-1}}(t_{k-1})\\
 &\!\!\!\!\!\!\cdot\frac{\partial\tc{\ep_k}{\de_k}(t_k,z)} {\partial z}|_{z=0}\cdot  \tc{\ep_{k+1}}{\de_{k+1}}(t_{k+1})\cdot\ldots\cdot\tc{\ep_g}{\de_g}(t_g).
\end{aligned}
$$
Since all genus one theta functions are even except the one with characteristic $\chars11$,  this partial derivative vanishes unless $\chars{\ep_j}{\de_j}=\chars{\ep_k}{\de_k}=\chars11$.

Thus expanding theta constants in Taylor series in $\tau_{jk}$ near $\calD_g$, {\em for $t_1,\ldots,t_g$ fixed}, the constant and linear terms are
\begin{equation}\label{eq:expand}\begin{aligned}
 &\tc\ep\de(\tau)=\tc{\ep_1}{\de_1}(t_1)\cdot\ldots\cdot\tc{\ep_g}{\de_g}(t_g)\\ &+\frac{1}{2\pi i}\!\!\! \sum_{j<k,\chars{\ep_j}{\de_j}=\chars{\ep_k}{\de_k}=\chars11}\!\!\!\!\!\!\!\!\tau_{jk}\cdot
 \frac{\partial\tc11(t_j,z)} {\partial z}|_{z=0}\cdot
 \frac{\partial\tc11(t_k,z)} {\partial z}|_{z=0}\cdot\!\!
 \prod_{m\ne j,k}\!\! \tc{\ep_m}{\de_m}(t_m),
 \end{aligned}
\end{equation}
and the full Taylor series expansion includes further monomials in $\tau_{jk}$ that are of total degree 2 or higher. Note that if no column $\chars{\ep_m}{\de_m}$ is equal to $\chars11$, then the Taylor series has a non-zero constant term, and zero linear term. If precisely two different columns $\chars{\ep_j}{\de_j}$ and $\chars{\ep_k}{\de_k}$ are equal to $\chars11$, then the Taylor series has zero constant term, and the linear term is a multiple of $\tau_{jk}$. Finally, if more than two columns are equal to $\chars11$, then both the constant and linear terms of the Taylor series are zero, and in fact the lowest order term has degree equal to half the number of columns equal to $\chars11$.

Furthermore, recalling Jacobi's triple product identity in genus 1:
\begin{equation}\label{eq:tripleproduct}
 \frac{\partial\tc11(t,z)}{\partial z}|_{z=0}=-\pi\tc00(t)\tc01(t)\tc10(t),
\end{equation}
the linear term can be written in terms of theta constants of $t_j$ and $t_k$.

\section{Poincar\'e relations}
The Poincar\'e relation in genus 4 was discovered by Poincar\'e~\cite{poincare}; it is an infinitesimal relation for period matrices of Jacobians near $\calD_g$ --- Poincar\'e calls it ``approximate identity''. Rather than thinking of it as being the lowest order terms of a suitable power series expansion, Poincar\'e derived it from something he called a ``translation surface''. Poincar\'e then stated that his proof in genus 4 could be easily generalized to higher genus, but gave no details. Garabedian~\cite{garabedian} proved the Poincar\'e relations for some special Riemann surfaces, in arbitrary genus.

Rauch in~\cite{rauchschpo} reproved the original Poincar\'e relation in genus 4 by expanding a suitable genus $4$ Schottky-Jung identity near $\calD_4$, and in~\cite[Appendix 2]{rafabook} asked whether Poincar\'e relations for any genus could be obtained in a similar way. While Fay~\cite{faytheta} proved the Poincar\'e relation in any genus, we give a direct new proof in the spirit of Rauch, obtaining it by expanding $S_{jk}$ (see~\Cref{rem:noRiemann} for a discussion of why using Riemann's quartic relations instead of $R_{jk}$ would not work).

The Poincar\'e relations are the following equations for the off-diagonal elements of the period matrix, for all $i<j<k<l$:
\begin{equation}\label{eq:Poin}
  (\tau_{ij}\tau_{jk}\tau_{kl}\tau_{li})^{1/2}\pm(\tau_{ik}\tau_{kl}\tau_{lj}\tau_{ji})^{1/2}\pm(\tau_{il}\tau_{lj}\tau_{jk}\tau_{ki})^{1/2}
=O(\ep^3).
\end{equation}
 Note that the Poincar\'e relations do not depend on the diagonal entries $t_m$ of the period matrix, which is a priori surprising.

Similar to the case of Schottky-Jung proportionalities, the signs of the square roots in the Poincar\'e relation may not be chosen globally, and to get a well-defined polynomial equation in the entries of the period matrix, one multiplies the relations~\eqref{eq:Poin} for all four possible choices of square roots --- this was explained by Igusa~\cite[p.167]{iguproblems}. As in the derivation of~\eqref{eq:Rconjugates}, if we denote the three terms in the Poincar\'e relation by $\sqrt{P_1},\sqrt{P_2},\sqrt{P_3}$, the resulting equation that is polynomial in the entries of the period matrix has the form
\begin{equation}\label{eq:Pconjugates}
  P_1^2+P_2^2+P_3^2-2P_1P_2-2P_1P_3-2P_2P_3=O(\ep^{9}).
\end{equation}

We prove Poincar\'e relations by expanding the factors of $S_{jk}$ near $\calD_g$.
\begin{thm}\label{thm:Poincareholds}
Let $C$ be any curve in $\calM_g$ sufficiently close to a union of $g$ elliptic curves. Then after an appropriate choice of $A$ and $B$ cycles, the period matrix $\tau$ of $J(C)$ satisfies all Poincar\'e relations~\eqref{eq:Poin}.
\end{thm}
In genus 4 there is a unique Poincar\'e relation, for the quadruple $(ijkl)=(1234)$, and we first present a streamlined version of Rauch's computation in~\cite{rauchschpo} deriving it from a Schottky-Jung identity.
\begin{proof}[Proof of~\Cref{thm:Poincareholds} in genus 4]
Applying the Schottky-Jung proportionalities for $\eta_4$, given by~\eqref{eq:SJetag}, to relation~\eqref{eq:Rgen3} gives the identity
$$
\begin{aligned}
0&=s_{34}:=\\
&(\tc{0 & 0 & 0 & 0}{0 & 0 & 0 & 0}\tc{0 & 0 & 0 & 0}{ 1 & 1 & 1 & 1}\tc{0 & 0 & 1 & 1}{0 & 1 & 0 & 0}\tc{0 & 0 & 1 & 1}{1 & 0 & 1 & 1}
\tc{1 & 1 & 0 & 0}{0 & 0 & 0 & 1}\tc{1 & 1 & 0 & 0}{ 1 & 1 & 1 & 0}\tc{1 & 1 & 1 & 1}{0 & 1 & 0 & 1}\tc{1 & 1 & 1 & 1}{1 & 0 &  1 &  0})^{1/2}\\
&\pm(\tc{1 & 0 & 1 & 0}{0 & 0 & 0 & 0}\tc{1 & 0 & 1 & 0}{1 & 1 & 1 & 1}\tc{1 &  0 & 0 & 1}{0 & 1 & 0 & 0}\tc{1 & 0 & 0 & 1}{1 & 0 & 1 & 1}\tc{0 & 1 & 1 & 0}{0 & 0 & 0 & 1}\tc{0 & 1 & 1 & 0}{1  & 1 & 1 & 0}\tc{0 & 1 & 0 & 1}{0 & 1 & 0 & 1}\tc{0 & 1 & 0 &1}{1 & 0 & 1 & 0} )^{1/2}\\
&\pm (\tc{0 & 0 & 0 & 0}{0 & 0 & 1 & 1}\tc{0 & 0 & 0 & 0}{1 & 1 & 0 & 0}\tc{0 & 0 & 1 & 1}{0 & 1 & 1 & 1}\tc{0 & 0 & 1 & 1}{1 & 0 & 0 & 0}
\tc{1 & 1 & 0 & 0}{0 & 0 & 1 & 0}\tc{1 & 1 & 0 & 0}{1 & 1 & 0 & 1}\tc{1 & 1  & 1 & 1}{0 & 1  & 1 & 0}\tc{1 & 1 & 1 & 1}{1 & 0 & 0 & 1})^{1/2}.
\end{aligned}
$$
Denoting $RR_1,RR_2,RR_3$ the three degree 8 monomials in theta constants appearing here, we now use the expansion~\eqref{eq:expand} for each theta constant involved in $s_{34}$, computing up to linear order terms in $\tau_{ab}$, for all $t_m$ fixed. Note that for each monomial $RR_i$ four of the theta characteristics involved have all columns being even $1$-characteristics, and the remaining four theta characteristics have precisely two columns equal to $\chars11$. As discussed in the previous section, for those 4 theta constants where all columns are even, the lowest degree term of the expansion is the constant term, while for the remaining 4 theta constants the lowest degree term is linear, and equal to a multiple of $\tau_{ab}$, where the columns $\chars{\ep_a}{\de_a}=\chars{\ep_b}{\de_b}=\chars11$. For example for $RR_1$ in the 6th characteristic involved we get the term with $\tau_{12}$, in the 7th characteristic --- the term with $\tau_{24}$, in the 4th --- $\tau_{34}$, and in the 8th --- $\tau_{13}$.  We thus compute the lowest degree term of the expansion of $RR_1$ near $\calD_4$ to be the product of these four linear terms, and the four constant terms from the expansion of the other theta constants. By checking that the number of times in each product $RR_i$ that each column $\chars{\ep_m}{\de_m}$ is equal to each of the even $1$-characteristics is equal to two, this gives for the lowest degree term
$$
\begin{aligned}
 RR_1=&(2\pi i)^{-4} \tau_{12}\tau_{24}\tau_{34}\tau_{13}\cdot \\ &\cdot\prod_{m=1}^4\Big((\tc00(t_m)\tc01(t_m)\tc10(t_m)\Big)^2\cdot\prod_{m=1}^4\left(\frac{\partial\tc{1}{1}(t_m,z)} {\partial z}|_{z=0}\right)^2.
\end{aligned}
$$
The lowest order terms of $RR_2$ and $RR_3$ are similar. They have exactly the same factor in theta constants and derivatives, while the entries of the period matrix that appear are
$$
\tau_{13}\tau_{23}\tau_{24}\tau_{14}
$$
from the 2nd, 6th, 7th, and 4th theta constants appearing in the product $RR_2$, and similarly
$$
\tau_{14}\tau_{34}\tau_{12}\tau_{23}
$$
for $RR_3$. Using Jacobi's triple product identity~\eqref{eq:tripleproduct}, we see that the overall theta factor in each of these is simply equal to the product
$$
 c:=\prod_{m=1}^4\prod_{\chars\ep\de\in K_1^+}\theta^4\chars\ep\de(t_m),
$$
which is non-zero for any $t_1,t_2,t_3,t_4$ in the upper half-plane. Up to this common factor, the square root of the lowest degree term of the expansion of each $RR_i$ is then equal to the corresponding summand in the Poincar\'e relation~\eqref{eq:Poin} for the quadruple $(1234)$. Thus altogether we have computed the lowest order term of the expansion:
$$
s_{34}=c \left(\pm(\tau_{34}\tau_{12}\tau_{24}\tau_{13})^{1/2}\pm( \tau_{13}\tau_{14}\tau_{23}\tau_{24})^{1/2}\pm (\tau_{12}\tau_{14}\tau_{23}\tau_{34})^{1/2}\right)+O(\ep^3).
$$
Recall now that in genus 4 the identity $S_{34}$ is obtained as a product of four factors of the form $s_{34}$, with different choices of signs. The expansion of each such factor near the diagonal is as above, with suitable choices of signs, and thus the expansion of the product of the four factors gives exactly the Poincar\'e's relation in its polynomial form~\eqref{eq:Pconjugates}.
\end{proof}
\begin{rem}\label{rem:etag}
Many other choices of $3$-characteristics for the genus 3 Riemann's quartic relation, instead of~\eqref{eq:Rgen3}, and different choices of $\eta$ would also yield a Schottky-Jung identity that would work in the above proof. One just needs to make sure that in each resulting $RR_i$ precisely four theta characteristics have all columns even, and the remaining four characteristics have precisely two columns equal to $\chars11$. For generalizing to higher genus, to be able to deal with the small Schottky locus rather than the big Schottky, we need to use the same two-torsion point $\eta$ for all proportionalities, and thus in our approach the 3rd, 4th,\dots, $g$'th columns of $\eta$ should all be equal, so that permuting them would leave $\eta$ invariant. The computationally simplest choice of such a two-torsion point could be $\chars{000\ldots 0}{110\ldots 0}$, but using a computer we checked that for all possible Riemann's quartic relations in genus 3, applying the Schottky-Jung proportionalities for the two-torsion point $\chars{0000}{1100}$ to them cannot give in genus 4 a Schottky-Jung identity where the columns of $RR_1,RR_2,RR_3$ satisfy this necessary combinatorial property. Thus our choice of $\eta_g$ is the simplest possible.
\end{rem}
We now generalize this computation to arbitrary genus; note that this of course uses the specifics of our choice of quartic identities $R_{jk}$ and of the corresponding Schottky-Jung identities $S_{jk}$.
\begin{proof}[Proof of~\Cref{thm:Poincareholds} for arbitrary genus]
The proof is again by computing the lowest order terms of the appropriate expansions, and crucially noticing that the cases when some of the $\ep_5,\ldots,\ep_g$ are equal to $1$ lead to higher order terms, as some of the theta constants involved in $s_{34}$ have then more $\chars11$ columns. Recall that $S_{34}$ is the product of $2^{3\cdot 2^{g-4}-1}$ terms of the form
$$
\begin{aligned}
s_{34}:=  &\sum_{\ep \in(\ZZ/2\ZZ)^{g-4}}
a_\ep\left(\tc{E & 0 & 0 & 0&\ep }{0 & 0 & 0 & 0&{\bf 0} }\tc{E & 0 & 0 & 0&\ep }{ 1 & 1 & 1 & 1&{\bf 1} }
\tc{E & 0 & 1 & 1&\ep }{0 & 1 & 0 & 0&{\bf 0} }\tc{E & 0 & 1 & 1&\ep }{1 & 0 & 1 & 1&{\bf 1} }\right.\cdot\\
&\left.\cdot\,\tc{1+E & 1 & 0 & 0&\ep }{0 & 0 & 0 & 1&{\bf 0} }\tc{1+E & 1 & 0 & 0&\ep }{ 1 & 1 & 1 & 0&{\bf 1} }
\tc{1+E & 1 & 1 & 1&\ep }{0 & 1 & 0 & 1&{\bf 0} }\tc{1+E & 1 & 1 & 1&\ep }{1 & 0 &  1 &  0&{\bf 1} }\right)^{1/2}\\
+&b_\ep\left(\tc{1+E & 0 & 1 & 0&\ep }{0 & 0 & 0 & 0&{\bf 0} }\tc{1+E & 0 & 1 & 0&\ep }{1 & 1 & 1 & 1&{\bf 1} }
\tc{1+E &  0 & 0 & 1&\ep }{0 & 1 & 0 & 0&{\bf 0} }\tc{1+E & 0 & 0 & 1&\ep }{1 & 0 & 1 & 1&{\bf 1} }\right.\cdot\\
&\ \ \left.\cdot\,\tc{E & 1 & 1 & 0&\ep }{0 & 0 & 0 & 1&{\bf 0} }\tc{E & 1 & 1 & 0&\ep }{1  & 1 & 1 & 0&{\bf 1} }
\tc{E & 1 & 0 & 1&\ep }{0 & 1 & 0 & 1&{\bf 0} }\tc{E & 1 & 0 &1&\ep }{1 & 0 & 1 & 0&{\bf 1} } \right)^{1/2}\\
+&c_\ep \left(\tc{E & 0 & 0 & 0&\ep }{0 & 0 & 1 & 1&{\bf 0} }\tc{E & 0 & 0 & 0&\ep }{1 & 1 & 0 & 0&{\bf 1} }
\tc{E & 0 & 1 & 1&\ep }{0 & 1 & 1 & 1&{\bf 0} }\tc{E & 0 & 1 & 1&\ep }{1 & 0 & 0 & 0&{\bf 1} }\right.\cdot\\
&\ \ \left.\cdot\,\tc{1+E & 1 & 0 & 0&\ep }{0 & 0 & 1 & 0&{\bf 0} }\tc{1+E & 1 & 0 & 0&\ep }{1 & 1 & 0 & 1&{\bf 1} }
\tc{1+E & 1  & 1 & 1&\ep }{0 & 1  & 1 & 0&{\bf 0} }\tc{1+E & 1 & 1 & 1&\ep }{1 & 0 & 0 & 1&{\bf 1} }\right)^{1/2},
\end{aligned}$$
for different choices of the signs $a_\ep,b_\ep,c_\ep$.
If $\ep=\bf 0$, then for all theta constants involved in the three corresponding summands in $s_{34}$, in columns $5,\ldots,g$ only characteristics $\chars00$ and $\chars01$ appear. Thus the lowest order term for the expansion near $\calD_g$ of each theta constant involved is simply equal to the lowest term of the expansion of the genus 4 theta constant with the first four columns as characteristics, times the suitable product of theta constants of $t_5,\ldots,t_g$. Since in each of the three summands each of the columns $5,\ldots,g$ takes each of the values $\chars00$ and $\chars01$ exactly 4 times, the lowest order term of the expansion near $\calD_g$ of each of the three terms with $\ep=\bf 0$ is equal to the expansion near $\calD_4$ of the corresponding term in genus 4, times the factor of $\prod_{m=5}^g\theta^2\chars00(t_m)\theta^2\chars01(t_m)$ --- which is the same for these three terms.

If any of the $\ep_5,\ldots,\ep_g$ are equal to 1, then in each of the three summands appearing in the expression for $s_{34}$ for such $\ep$, four of the theta characteristics --- those that have ${\bf 1}$ on the bottom --- will have extra columns equal to $\chars11$. Thus the lowest order term for the expansion of such summand near $\calD_g$ would be of higher order, as explained after formula~\eqref{eq:expand}.

Hence in the expansion  of $s_{34}$ near $\calD_g$ the only terms of degree 4 in $\tau_{ab}^{1/2}$ arise from the case $\ep_5=\ldots=\ep_g=0$, and they are equal to the expansion of the expression for $s_{34}$ in genus $4$, times $\prod_{m=5}^g\theta^2\chars00(t_m)\theta^2\chars01(t_m)$. Thus the lowest order term of the expansion near $\calD_g$ of $S_{34}$ --- which is a product of $2^{3\cdot 2^{g-4}-1}$ terms of the form $s_{34}$, with different choices of signs --- is equal to the lowest order term of the expansion of $S_{34}$ in genus 4, taken to power $2^{3\cdot 2^{g-4}-3}$ (corresponding to the choices of signs $a_\ep,b_\ep,c_\ep$ for all $\ep\ne{\bf 0}$, which do not change the lowest order term of $s_{34}$), times a power of $\prod_{m=5}^g\tc00(t_m)\tc01(t_m)$. Since this product of genus one theta constants is never zero, the vanishing of the lowest order term of the expansion near $\calD_g$ of $S_{34}$ implies the vanishing of the power of the lowest order term that appears in the Poincar\'e relation for $(1234)$. Thus the Poincar\'e relation for the quadruple $(1234)$ holds, up to terms of higher order.
By interchanging the columns $(1234)$ and $(ijkl)$ of the characteristics involved, we note that the correspondingly permuted Schottky-Jung identity implies the Poincar\'e identity for any given quadruple $(ijkl)$.
\end{proof}
\begin{rem}\label{rem:noRiemann}
The proof above shows why applying the Schottky-Jung proportionalities for $\eta_g$ to Riemann's quartic relations directly would not work. Indeed, if one were to take the genus 3 Riemann's quartic relation~\eqref{eq:Rgen3} and extend the genus 3 characteristics $\chars\al\be=\chars{011}{100}, \chars\sigma\mu=\chars{100}{001}$ to genus $g-1$ simply by zero characteristics (or in fact in any other way), then instead of $R_{34}$, where only the sum over top genus $(g-4)$-characteristics $\ep$ is taken, we would have a quartic identity where the sum over all $\chars\ep\de\in K_{g-4}$ is taken. Then the lowest degree terms of the expansion of $s_{34}$ near the diagonal would arise if no column $\chars{\ep_m}{\de_m}$ or $\chars{\ep_m}{\de_m+1}$ is equal to $\chars11$; thus the lowest degree terms would arise from the cases when all $\chars{\ep_m}{\de_m}$ are equal to $\chars00$ or $\chars01$.

But then for two monomials appearing in $s_{34}$, where all columns are the same, except say for $\chars{\ep_g}{\de_g}=\chars00$ in one monomial versus $\chars{\ep_g}{\de_g}=\chars01$ in the other, the lowest degree term of the expansion is exactly the same. Indeed, these two lowest degree terms give the same expression in $\tau_{ab}$, for $1\le a<b\le 4$, times the same product of theta constants in variables $t_5,\ldots,t_{g-1}$, but also times the same factor of $\theta^2\chars00(t_g)\theta^2\chars01(t_g)$ in both cases --- as in each case both of these characteristics appear four times under the square root. However, as the signs of the individual square roots of degree 8 monomials cannot be determined, it could be that these two lowest degree terms simply cancel. Thus it could be that the desired lowest degree term of $s_{34}$ would cancel out, and the argument above would fail. Using Fay's (generalized) Riemann theta formula, or equivalently the doubling trick (see~\Cref{rem:doubling}) to obtain $R_{34}$ allows us to overcome this crucial difficulty.
\end{rem}
\bigskip
For arbitrary genus $g$, there are $\binom{g}{4}$ Poincar\'e relations. In particular, already for $g=5$ only 3 out of the 5 Poincar\'e relations can be locally independent for dimension reasons. Following Rauch's ideas~\cite[p.228]{rafabook} for $g=5$, one can exhibit for any $g\ge 4$ a collection of
$$
 \tfrac{(g-3)(g-2)}{2}=\tfrac{g(g+1)}2-(3g-3)=\dim\calA_g-\dim\calJ_g
$$
Poincar\'e relations that are locally functionally independent.
\begin{prop}\label{lm:Pindep}
For any $g\ge 4$ the $(g-3)(g-2)/2$ Poincar\'e relations \eqref{eq:Poin}, corresponding to the quadruples of the form $(12jk)$ for all $3\le j<k\le g$, are functionally independent in a neighborhood of a generic $\tau\in\calD_g$.

In particular, the codimension of the locus in $\calA_g$ determined by these Poincar\'e relations is locally equal to $\tfrac{(g-3)(g-2)}2$ near such $\tau$, and the dimension at $\tau$ of the locus of period matrices satisfying these Poincar\'e relations is equal to $3g-3$.
\end{prop}
While it is possible to give a direct proof of this statement by showing that locally given   $t_1,\ldots,t_g $, $\tau_{12},\tau_{13}, \tau_{23},\ldots,\tau_{1g},\tau_{2g}$, all the $\tau_{ab}$ can be determined from this set of Poincar\'e relations, this Proposition follows from the proof of local functional independence of the corresponding $S_{jk}$, given in the next section.

\section{Schottky-Jung and Poincar\'e: proof of the main theorem}
In this  section we  will prove the main theorem, by showing that the lowest order terms of the expansion of $S_{jk}$ near $\calD_g$ give a collection of functionally independent   relations.
\begin{proof}[Proof of the main theorem]
We first consider the genus 4 case. In this case we only have one identity $S_{34}$, and~\Cref{lm:SR} yields $S_{34}=SJ^{\eta_4}(R_{34})$.
The full Taylor series expansion of  $S_{34}$ in the variables $\tau_{jk}$, near the point $\tau=\diag(t_1,t_2,t_3,t_4)\in\calD_g$, for all $t_m$ fixed, is a series such that its lowest degree term has degree $8$ in the $\tau_{jk}$. By the computations in the previous section, this lowest degree term is equal to a non-zero multiple of the symmetrization ~\eqref{eq:Pconjugates} of the Poincar\'e relation. Thus the whole series is not identically zero, and consequently its zero locus is of codimension one in $Th(\calA_g(4,8))$ --- hence of dimension 9. Since the zero locus of $S_{34}$ contains the irreducible 9-dimensional locus $\calJ_4$, it follows that $\calJ_4$ must be an irreducible component of this zero locus.

For arbitrary genus $g$, by the proof of~\Cref{thm:Poincareholds} we know that the lowest order terms of the expansions of $S_{jk}$ near $\calD_g$ are non-zero multiples of powers of the Poincar\'e relations for quadruples $(12jk)$. To prove that $S_{jk}$ are functionally independent near a generic point of $\calD_g$, we can follow  the idea of the argument given in~\cite[pp.~227ff]{rafabook} (which is possible now that we have found Schottky-Jung identities whose expansions give Poincar\'e relations, and now that we have handled the issue of signs, and ascertained suitable non-vanishing correctly).

Consider the  jacobian  matrix of derivatives of the   $S_{jk}$ with respect to the  variables $\tau_{ab} $ with  $3\leq a<b\leq g$, evaluated very close to a generic point of $\calD_g$ --- i.e.~for $\tau_{ii}=t_i$ generic, for all $1\le i\le g$, and for $0<|\tau_{ab}|<\ep\ll 1$, for any $1\leq i<j\leq g$.  To compute $\partial S_{34}/\partial \tau_{ab}$ for $3\leq a<b\leq g$, note that the lowest degree term is always zero except for the case of $\partial S_{34}/\partial \tau_{34}$. Since each $S_{jk}$ is obtained from $S_{34}$ by permuting the columns, it follows that to lowest order the only non-zero partial derivative is $\partial S_{jk}/\partial \tau_{jk}$. Thus the jacobian matrix is diagonal, plus terms of higher order in the off-diagonal entries $\tau_{ab}$ of the period matrix. Since we have assumed that all $|\tau_{ab}|<\ep$, the determinant of this jacobian matrix is equal to the jacobian of the lowest order diagonal matrix, plus some $O(\ep)$. Since the determinant of the diagonal matrix is non-zero, for sufficiently small $\ep$ it thus follows that the jacobian determinant $\det (\partial S_{jk}/\partial\tau_{ab})$ is non-zero, and thus that the equations $S_{jk}$ are functionally independent.
\end{proof}

\end{document}